\documentclass[12pt, a4paper, oneside]{article}
\usepackage[utf8]{inputenc}
\usepackage{graphicx}

\usepackage{titlesec}
\graphicspath{{images/}}
\usepackage[a4paper,top=15mm,bottom=15mm,width=165mm,bindingoffset=6mm]{geometry}
\usepackage{amssymb}
\usepackage{nomencl}
\makenomenclature
\usepackage{caption}
\usepackage{subcaption}
\usepackage{epigraph}
\usepackage{amsthm}
\usepackage{pdfpages}
\usepackage{setspace}
\usepackage{nameref}
\usepackage{comment}
\usepackage{enumerate}  
\RequirePackage[breaklinks,colorlinks]{hyperref}
\usepackage{enumitem}  
\usepackage{todonotes}
\usepackage{hyperref}
\usepackage{soul}

\usepackage{mathtools}

\RequirePackage[breaklinks,colorlinks]{hyperref}

\usepackage{url}
\usepackage[parfill]{parskip}
\usepackage{stackengine}
\usepackage{xfrac}
\usepackage{mathtools}

\usepackage{braket}
\usepackage[nottoc]{tocbibind}

\usepackage[square,numbers]{natbib}
\usepackage{titlesec}

\newtheorem{myassumption}{\,\,\,\,\,\, \normalfont\it Assumption}

\newtheorem{definition}{Definition}
\newtheorem{assumption}{Assumption}
\newtheorem{corollary}{Corollary}
\newtheorem{theorem}{Theorem}
\newtheorem{remark}{Remark}
\newtheorem{lemma}{Lemma}

\hypersetup{colorlinks=true, linkcolor=black}

\DeclareMathOperator\supp{supp} 
\DeclareMathOperator\spark{spark}

\DeclareMathOperator*{\argmin}{arg\,min}

\definecolor{myblue}{rgb}{0,0,0.5}

\numberwithin{equation}{section}

\usepackage{etoolbox}
\makeatletter
\patchcmd{\endmyassumption}{\@endpefalse}{}{}{}
\makeatother

\usepackage{authblk}
\providecommand{\keywords}[1]{\textbf{\textit{Keywords:}} #1}

\date{October, 2025}
\title{Unique Reconstruction From Mean-Field Measurements}

\author[1]{Narcicegi Kiran}

\author[1]{Tiago Pereira}

\affil[1]{Instituto de Ci\^encias Matem\'aticas e Computac\~ao, Universidade de S\~ao Paulo, São Carlos, 13566-590, Brazil}

\begin{document}
\maketitle

\begin{abstract}
We address the inverse problem of reconstructing both the structure and dynamics of a network from mean-field measurements, which are linear combinations of node states. This setting arises in applications where only a few aggregated observations are available, making  network inference  challenging. We focus on the case when the number of mean-field measurements is smaller than the number of nodes. To tackle this ill-posed recovery problem, we propose a framework that combines localized initial perturbations with sparse optimization techniques. We derive sufficient conditions that guarantee the unique reconstruction of the network’s adjacency matrix from mean-field data and enable recovery of node states and local governing dynamics. Numerical experiments demonstrate the robustness of our approach across a range of sparsity and connectivity regimes. These results provide theoretical and computational foundations for inferring high-dimensional networked systems from low-dimensional observations.\\

\end{abstract}
\newpage
\tableofcontents

\section{\label{sec:level1}Introduction}
Networks of dynamical systems model the behavior of complex systems in biology \cite{winfree1967biological}, neuroscience \cite{ermentrout2010mathematical}, and physics~\cite{newman2018networks}. By encoding interaction patterns among components, these models help explain emergent collective phenomena such as synchronization and critical transitions~\cite{eroglu2017sync,scheffer2009early,liu2024early}. The emergence of such dynamics is dictated by the network structure: connectivity changes can induce qualitative behavior shifts. This has motivated great interest in identifying  interaction patterns in systems where disruptions may trigger pathological states \cite{thijs2019epilepsy}.

Yet, the network structure is not directly observable in most real-world systems \cite{hawrylycz2012anatomically}. Due to experimental constraints, we rarely have direct access to information about who interacts with whom. A central challenge is to recover the hidden structure using only observable outputs \cite{
wang2018inferring, novaes2021recovering, 
tan2023backpropagation, delabays2025hypergraph}. A common assumption in network reconstruction is that the full state of the system is accessible, that is, time series data from all nodes are available. Techniques based on least squares \cite{brunton2016discovering}, compressive sensing \cite{candes2005decoding, candes2006stable}, Bayesian estimation \cite{stankovski2017coupling}, and exploiting the ergodic properties of the system \cite{eroglu2020revealing} have been developed to recover both the governing equations and thereby the interaction structure. When combined with assumptions about the network dynamics, these approaches can  guarantee the uniqueness of the reconstructed structure \cite{pereira2023robust}.

Accessing the states of all units is unrealistic in many practical applications. Limitations in spatial resolution, signal quality, or system size  restrict the scope of available measurements. For instance, in neuroscience, it remains infeasible to monitor every neuron in a brain region \cite{zhao2023tracking}; in power grids or sensor arrays, measurements are limited to aggregated or strategically sampled data \cite{liu2013observability, haber2017state, casadiego2017model, schaeffer2018extracting, ding2024deep, berry2025reconstruction}. In these contexts, one considers mean-field observations that are aggregated sums of node states. Such measurements arise in experiments where the recorded signals reflect the collective activity of many units and are often the only accessible data \cite{buzsaki2012origin}.

We address the case where the governing dynamics are unknown and the system is observed exclusively through a mean-field measurement. That is, only a linear combination of groups of node states is accessible and not the node states themselves. This mean-field measurement is common in applications and gives rise to further mathematical challenges during recovery. The measurements are not one-to-one when the number of mean fields is smaller than the number of nodes.   The central question is whether and how structural or dynamical information about the network can be recovered from such compressed observations.

\subsection{This Paper's Contributions:}

We tackle two levels of network reconstruction: topological reconstruction aiming to recover the network’s adjacency matrix, and full network reconstruction aiming to recover the underlying governing equations.

The key idea behind our approach is to establish a one-to-one correspondence between the support of the system states and the network's adjacency matrix. This correspondence is enabled by a mechanism we call pinching. We consider a class of initial conditions that begin as localized pulses at individual nodes, called pinching initial conditions. By measuring the resulting mean-field responses over time, we isolate the influence of each node's local neighborhood on the global dynamics.

In the topological reconstruction setting, we show that the graph’s maximum out-degree governs the number of required mean-field measurements. To formalize this, we introduce two key notions: the weak topological reconstruction condition (wTRC) and the strong topological reconstruction condition (sTRC). Both conditions ensure the uniqueness of the reconstructed topology. The wTRC is based on support-level information, while the sTRC allows for efficient recovery via convex optimization.

We prove that under the wTRC, topological reconstruction is guaranteed when the number of mean-field measurements is at least twice the maximum out-degree of the graph. The sTRC, in turn, guarantees the uniqueness of the recovered states through a convex program when the mean-field measurement matrix satisfies certain conditions. Importantly, our results imply that for many practical cases, such as sparse or random networks, the number of required mean-field measurements can be significantly smaller than the network size. For instance, in sparse random graphs with $N$ nodes, the number of measurements may scale as $\log^2 N$.

We propose a two-stage strategy for full network reconstruction: (i) uniquely recover the system states, and (ii) infer the local interaction dynamics from these states. For inference, we assume that each node has a candidate dictionary containing the spanning functions that exactly represent its dynamics. We demonstrate that the number of measurements depends on the interaction between the local dynamics and the network's sparsity. This trade-off is explicit in our theoretical bounds to achieve structural and dynamical recovery.
We validate our theoretical results through numerical experiments, demonstrating the accuracy and efficiency of the proposed reconstruction methods in various regimes.

This paper is organized as follows. Section \ref{setting-sec} presents the model setup and introduces the notion of mean-field measurements under pinching initial conditions. Section \ref{trc-section} is devoted to topological reconstruction: we define the weak and strong topological reconstruction conditions (wTRC and sTRC), and provide the theoretical guarantees for uniqueness and recovery. Section \ref{full-rec-sect} addresses full network reconstruction, detailing the two-stage recovery approach and analyzing the interplay between network sparsity and local dynamics. Section \ref{numerics-sec} contains numerical experiments that illustrate and validate the theoretical results. We conclude in Section \ref{proof-sec} providing the proofs of the main results.

\section{Setting and Main Problem} \label{setting-sec} 
Our graph notation follows the Ref. \cite{west2001introduction}. Let \(G\) be an unweighted directed graph without self-loops, and \([N]\) be the set \(\{1,\hdots,N\}\) for \(N\in \mathbb{N}\). The out-degree \( d_q \) of a vertex \( q \in [N] \) is the number of edges originating from vertex \( q \). The maximum out-degree of the graph \( G \) is denoted by \( \Delta(G) \). In addition, the first-level set \(L_1(q)\) consists of the vertices that receive edges from \(q \in [N] \). Next, we introduce the setting and the main problem.

Consider the network given by the discrete-time dynamical system 
\begin{equation}\label{eqdiffus}
\begin{split}
&{x}_{i}(t+1) = f_i\left(x_{i}(t) \right) +\sum_{j=1}^{N} A_{ij} h_{ij} (x_{i}(t), x_{j}(t)), \qquad \text{for every } i \in [N],
\end{split}
\end{equation}
where $A=(A_{ij})$ is the adjacency matrix of the graph \(G\) where $A_{ij}=1$ if $i$ receives input from $j$ and zero otherwise, \(f_i:\mathbb{R} \rightarrow \mathbb{R}\) describes the isolated dynamics, \(h_{ij}:\mathbb{R} \times \mathbb{R} \rightarrow \mathbb{R}\) is the coupling function, and \(x_i\in \mathbb{R}\) represents the state of the $i$th vertex of $G$. We call $x^T=(x_1,\cdots, x_N)$ the state vector. We make the following: 

\begin{assumption}\label{isolated_ass} 
$f_i(0)=0$ for every $i\in[N]$.
\end{assumption}

\begin{assumption}\label{assumption_about_coupling}
$    h_{ij}(0, v) \neq 0 \text{  for every } v \in B_{\delta}(0)\setminus \{0\}  ,
$
for sufficiently small $\delta$, and
$h_{ij}(0,0) = 0,$
 for all \(i,j \in [N]\).
\end{assumption}

The existence of a resting state, that is, the zero vector, is then guaranteed. Diffusive interaction satisfies our assumption \ref{assumption_about_coupling}.
We are interested in the case where the state vector $x\in \mathbb{R}^N$ is generated by a localized pulse. 
\begin{definition}
We regard the following initial condition
\begin{equation}
x_i^q(0)=\epsilon_q\delta_{iq}
\end{equation}
 as \(q\)-pinching where \(q\in [N]\), \(\delta_{iq}\) is the Kronecker delta, and \(\epsilon_q\in B_{\delta}(0)\setminus \{0\} \), where $\delta$ is sufficiently small. We denote a state vector with \(q\)-pinching as  \(x^q(t)\) and \(i\)th component of the state vector with \(q\)-pinching as  \(x_i^q(t)\). 
\end{definition}

Next, we take mean-field measurements as a linear combination of the states $\{x_i\}_{i=1}^N$.  We consider $P$ such mean-field measurements as follows:
 
\begin{assumption}\label{measurement_matrix}
  We assume that $\phi\in \mathbb{R}^{P\times N }$  where $P<N$, and refer to it as the measurement matrix. The mean-field measurements $y$ of size $P$ is generated by 
\begin{equation}\label{experiment}
    y=\phi x,
\end{equation}
where $x\in\mathbb{R}^N$ is the state vector, $\phi$ is the measurement matrix.\footnote{In the literature, $\phi$ is sometimes called the encoder \cite{baraniuk2008simple} or the information operator \cite{donoho2006compressed}.} 
\end{assumption}
Here, $\phi$ and $y$ are known and the goal is to recover the vector $x$.
Let $x^q(t)$ be the iterated state vector of the $q$-pinching initial condition $x^q(0)$ at iteration step $t$. We denote the mean-field measurements of size $P$, generated by $x^q(t)$, as $y^q(t)$, i.e., $y^q(t) = \phi_q x^q(t)$. The measurement matrices don't need to be distinct.

\textbf{Main Problem:} We address the inverse problem where both the network and the state vectors are unknown, while $\{\phi_q\}_{q=1}^N$ and the mean-fields $\{y^q(1)\}_{q=1}^N$ are known.\footnote{This problem is  referred to as an inverse problem or a decoding problem.} Our objective is to recover the network from the mean-fields.

Reconstructing a network from mean-field measurements is an ill-posed problem. The main challenge lies in the ambiguity introduced by the non-injectivity of the measurement matrix when $P<N$, thus different state vectors can yield the same observation. The reconstruction can lead to multiple distinct network structures.

In the following chapters, we explore how to ensure the unique reconstructability of the network for the given number of mean-field measurements. We propose a twofold methodology: first, we focus on recovering the state vectors and adjacency matrix $A$. Second, we recover $(f,h)$ (see Figure \ref{scheme}).

\begin{figure}[h!]
    \centering
\includegraphics[width=1\textwidth]{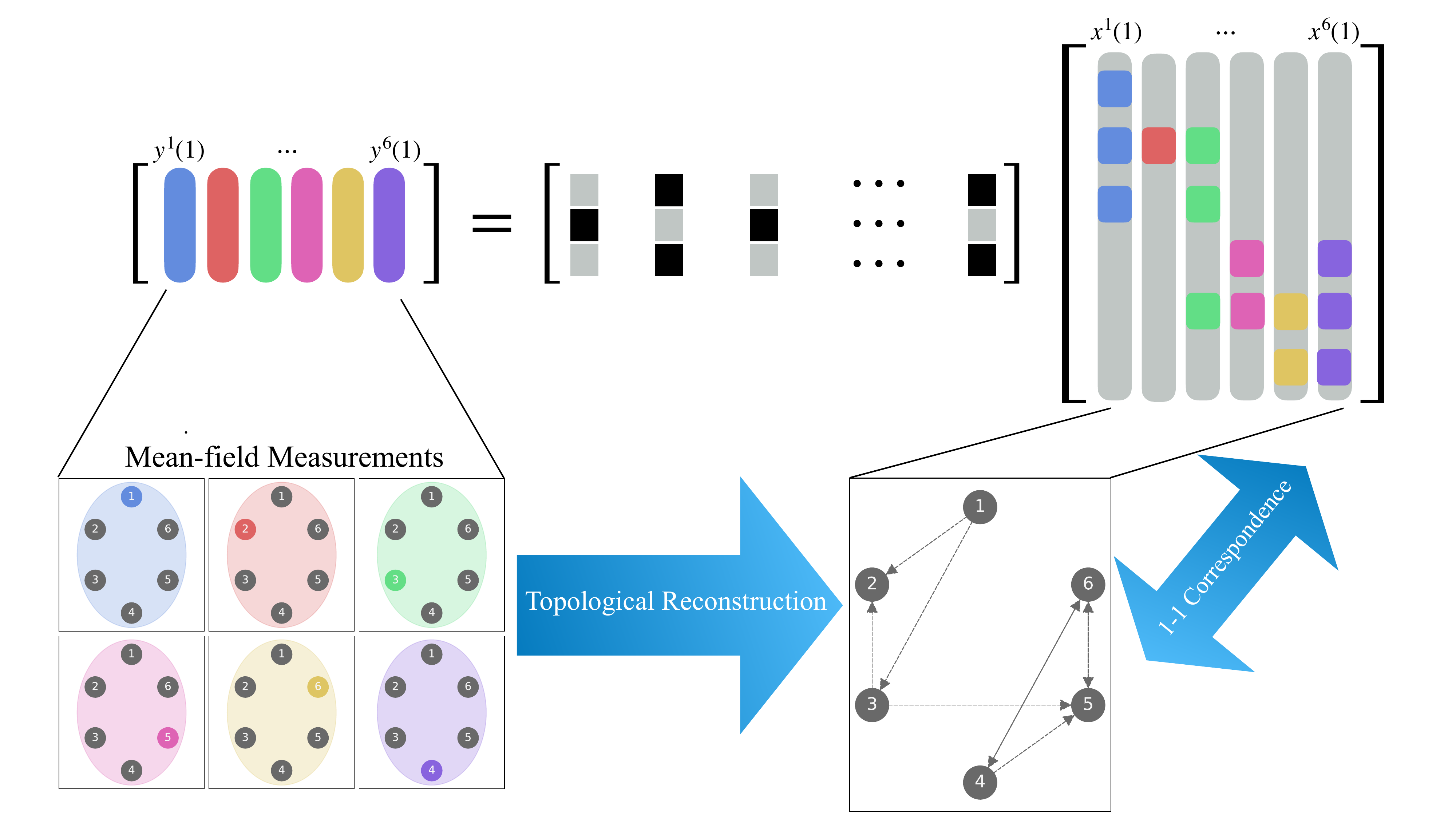}
\caption{ The light-colored ellipses represent the mean-field measurements. A set of mean-field measurements, generated through $q$-pinching initial conditions, is provided, while the state vectors remain unknown.
}
\label{scheme}
\end{figure}

\section{Topological Reconstruction from Mean-Fields}\label{trc-section}

The failure of unique reconstruction from mean-field measurements stems from the non-injective nature of the measurement matrices. One approach to obtain uniqueness is to ensure injectivity within the subspace of sparse state vectors. Roughly speaking, a vector is sparse when many of its entries are zero. Consider the support of a vector, defined as
\begin{equation}
    \supp(x):=\{i\mid x_i\neq 0\}, \quad \text{and} \quad 
\|x \|_0 := \#\{i   \mid x_i \neq 0\}.
\end{equation}

We say that a vector is $s$-sparse when \(\| x\|_0 \le s\). 
Using sparsity we can overcome the lack of injectivity. This is captured in the concept of spark.

\begin{definition}
The spark of a matrix $\phi\in \mathbb{R}^{P\times N}$ where $P<N$, is defined as the size of the smallest linearly dependent subset of its columns, which is given by
\begin{equation}
    \spark(\phi)= \min\{\|x\|_0: \phi x= 0, x\neq 0\}.
\end{equation}
If $\spark(\phi)=P+1$, it is called full-spark \cite{alexeev2012full}.
\end{definition} 
If $x\in \mathbb{R}^N$ is $s$-sparse, then  the inverse problem \(\phi x = y\) has a unique $s$-sparse solution provided that every set of $2s$ columns of $\phi$ is linearly independent. 
Thus, if $2\|x\|_0<\spark(\phi)$, then $x$ can be uniquely recovered (see Lemma \ref{spark_sparsest}).
We will show that, using these concepts, we can uniquely recover the states from mean-field measurements and the network topology.  We define the topological reconstruction as follows:  

\begin{definition} Recovering the adjacency matrix \(A\) of the underlying graph of the network $\mathcal{G}$ is called the topological reconstruction. 
\end{definition}
\begin{remark}\label{trc}
Notice that \(\{L_1(q)\}_{q=1}^N\) uniquely determines the adjacency matrix \(A\) of the underlying graph of the network $\mathcal{G}$ since \(L_1(q)\) consists of non-zero elements of the \(q\)th column of the adjacency matrix, i.e.,
\begin{equation}
    A_{ij}=\begin{cases}
        1 \text{ if } i\in L_1(j),\\0 \text{ if } i\not\in L_1(j),
    \end{cases} \text{ for all }i,j\in [N] .
\end{equation}
\end{remark}

We present two topological reconstruction conditions, referred to as the strong and weak conditions. This nomenclature reflects the distinction that the weak condition is derived from an abstract argument for unique existence, while the strong condition is computationally tractable through convex optimization. 

\subsection{Weak Topological Reconstruction Condition (wTRC)} \label{wtrc}

\begin{theorem}\label{topologicalcondition}(wTRC)
Let $\mathcal{G}=(G,f,h)$ be a network dynamical system of size $N$ with maximum out-degree $\Delta(G)$, satisfying Assumptions \ref{isolated_ass} and \ref{assumption_about_coupling}. Let $P>2\Delta(G)+1$, and let $\phi_q\in \mathbb{R}^{P\times N}$ be full-spark measurement matrices for each $q\in [N]$. Assume that mean-fields
\begin{equation}
    y^q(1)=\phi_q x^q(1)
\end{equation}        
 are known for all $q\in[N]$. Then, the network is uniquely topologically reconstructible from mean-field measurements. That is,
 given $\{y^q(1)\}_{q=1}^N$, for each $q\in [N]$, there exists a unique solution $x_*^q(1)$ to 
\begin{equation}\label{P0} 
 \min\|\tilde{x}^q(1)\|_0 \text{  subject to } y^q(1) = \phi_q \tilde{x}^q(1),
\end{equation}
such that 
\begin{equation}
    x_*^q(1) =x^q(1).
\end{equation}  
Moreover, there exists a one-to-one correspondence between $\supp(x^q(1))$ and $L_1(q)$.
\end{theorem}

This result investigates whether topological reconstruction from mean-field measurements is possible, providing an affirmative answer to this question and showing that 
the network topology can be reconstructed from the support of these vectors. 
Moreover, from an applied perspective, our result provides a framework for determining the minimum number of mean-field measurements needed for unique reconstruction, offering guidelines for experimental design and data collection.
\begin{remark}
    Since $\phi\in \mathbb{R}^{P\times N}$ is full-spark, the existence of solutions follows from the fact that full-spark implies full row rank and the column space $\mathrm{Col}(\phi)$ spans all of  $\mathbb{R}^P$, ensuring that at least one solution exists for any  $y \in\mathbb{R}^P$.
\end{remark}
\begin{corollary} \label{first_corol} (wTRC)
Let \(\{\phi_q\}_{q=1}^N\) be measurement matrices where each $\phi_q \in \mathbb{R}^{P \times N}$ is a Gaussian random matrix whose entries are i.i.d. Gaussian with mean zero and variance \(1/P\).
Then, with probability one, the matrices \(\{\phi_q\}_{q=1}^N\) 
satisfy the full-spark property. Consequently, for any network dynamical system satisfying the assumptions of Theorem \ref{topologicalcondition}, the network is uniquely topologically reconstructible from mean-field measurements almost surely.
\end{corollary}

We illustrate our result for Erd\H{o}s-Rényi (ER) graphs. We denote by \( G(N, p) \) an undirected Erd\H{o}s-Rényi random graph with \( N \) vertices, where each edge is included independently with probability \( p \). The expected degree \( \mathbb{E}[d] \) of such a graph is  \( (N - 1)p \). 
We show that for the Erd\H{o}s-Rényi model, the minimum number of required mean-field measurements scales logarithmically with the network size \( N \). 

\begin{corollary}\label{er_theorem}
Let \(\mathcal{G} = (G, f, h)\) be a network dynamical system of size \(N\) satisfying Assumptions \ref{isolated_ass} and \ref{assumption_about_coupling}, where \(G(N, p)\) is an Erd\H{o}s-Rényi random graph, and $N$ is sufficiently large. Let \(\phi_q \in \mathbb{R}^{P \times N}\) be full-spark measurement matrices for each \(q \in [N]\), with \(P(N)\) specified in each regime below. Assume that the mean-fields
\begin{equation}
    y^q(1) = \phi_q x^q(1)
\end{equation}
are known for all \(q \in [N]\). Fix $\epsilon \in (0,1)$.  
\begin{enumerate}
    \item (Supercritical)
    Suppose $p = \frac{\log N}{N}(1 - \epsilon)$.  
    If there is $c_1>0$ such that 
    \begin{equation}
        P(N) \ge c_1 \log N, \nonumber
    \end{equation}
    then asymptotically almost surely the networks in G(N,p) are uniquely topologically reconstructible from mean-field measurements.

    \item (Connected)
    Suppose $p = \frac{\log N}{N}(1 + \epsilon)$.
    If there is $c_1>0$ such that 
\begin{equation}
     P(N)  \ge c_1 \log^2 N, \nonumber
\end{equation}
    then asymptotically almost surely the networks in G(N,p) are uniquely topologically reconstructible from mean-field measurements.
\end{enumerate}
\end{corollary}

These cases demonstrate that the number of required mean-field measurements grows only logarithmically with the network size \( N \). For sparse and large-scale networks, a small fraction of measurements may suffice to ensure unique topological reconstruction with high probability, provided that the wTRC condition is satisfied.

\subsection{Strong Topological Reconstruction Condition (sTRC)}

In Section \ref{wtrc}, we showed the weak topological reconstruction. The maximum degree ensuring the topological reconstruction condition is deterministic and applies to mean-field measurements that are the full-spark. However, finding the solution to Eq. (\ref{P0})  is known to be NP-hard  \cite{natarajan1995sparse}. 
The sTRC pursues two main objectives: (1) reformulating the uniqueness condition into a computationally tractable framework, and (2) reducing the number of required mean-field measurements.
%
%

The restricted isometry property plays a key role:
\begin{definition}\label{RIPdef}
 The isometry constant \(\delta_s(\phi)\) is the smallest number such that 
\begin{equation}\label{RIP}
    (1-\delta_s(\phi)) \|x\|_2^2\leq\|\phi x\|_2^2 \leq (1+\delta_s(\phi))\|x\|_2^2
\end{equation}
    holds for all \(s\)-sparse vectors \(x\), \cite{candes2008restricted}. It is called \(\phi\) has restricted isometry property (RIP) when the inequality (\ref{RIP}) holds for given \(s\) and $\delta_s(\phi)\in (0,1).$ 
\end{definition}

Our main result is as follows:

\begin{theorem}\label{P1-main} (sTRC) Let $\mathcal{G}=(G,f,h)$ be a network dynamical system of size $N$ with maximum out-degree $\Delta(G)$, satisfying Assumptions \ref{isolated_ass} and \ref{assumption_about_coupling}. Let $\phi_q\in \mathbb{R}^{P\times N}$ be measurement matrices for each $q\in [N]$ such that $ \delta_{2(\Delta(G)+1)}(\phi_q)<\sqrt{2}-1$. Assume that mean-fields
\begin{equation}
    y^q(1)=\phi_q x^q(1)
\end{equation}        
 are known for all $q\in[N]$. Then, the network $\mathcal{G}$ is uniquely topologically reconstructible. That is, given $\{y^q(1)\}_{q=1}^N$, for each $q\in [N]$, there exists a unique solution $x_*^q(1)$ to 
\begin{equation}\label{P1} \min \|\tilde{x}^q(1)\|_1 \text{  subject to }  y^q(1) = \phi_q \tilde{x}^q(1),
\end{equation}
such that 
\begin{equation}
    x_*^q(1) =x^q(1).
\end{equation}
Moreover, there exists a one-to-one correspondence between $\supp(x^q(1))$ and $L_1(q)$.
\end{theorem}

The sTRC derived under exact (noiseless \(y^q(1) = \phi_q \tilde{x}^q(1)\)) conditions is robust under small perturbations \cite{donoho2006most2, candes2008restricted}, thereby supporting the practical relevance of the sTRC\footnote{We investigate this robustness by evaluating performance under varying the recovery thresholds.}.

\begin{remark}
Let $y^q(1)=\phi_q x^q(1) +z^q(1) $, where $z^q(1)$ is an unknown noise term satisfying $\|z^q(1)\|_2\leq \xi$. Suppose the measurement matrix $\phi_q$ satisfies the restricted isometry condition $\delta_{2(\Delta(G)+1)}(\phi_q)<\sqrt{2}-1$. Then, the solution $x_*^q(1)$ to:
\begin{equation}
    \min\|\tilde{x}^q(1)\|_1 \text{ subject to } \|y^q(1)-\phi_q \tilde{x}^q(1)\|_2\leq \xi,
\end{equation}
satisfies the stability bound:
\begin{equation}
    \|x^q(1)-x_*^q(1)\|_2\leq K\xi,
\end{equation}
for some positive constant $K$ \cite{candes2008restricted}. This implies that all feasible minimizers lie within an $\ell_2$ ball of radius $K\xi$ centered at the ground truth $x^q(1)$; in this sense, the solution is unique up to a precision threshold dictated by the noise level. Moreover, the one-to-one correspondence between $\supp(x^q(1))$ and $L_1(q)$ persists.
\end{remark}

Even though sTRC is deterministic, finding deterministic RIP-satisfying matrices remains an open problem~\cite{tao}. We focus our attention on the case where the measurement matrices are Gaussian random matrices. This case is relevant in many applications. The following remark together with Theorem \ref{P1-main} implies Corollary \ref{gaussian}. A formal proof of this corollary is omitted for brevity.

\begin{remark}
    E. Cand\`es et al. showed that  for Gaussian matrices whose entries are i.i.d. Gaussian with zero mean and the variance $1/P$, if
\begin{equation}
    s\leq \frac{c_1P}{\log(N/P)},
\end{equation}
where $c_1$ is a positive constant, $s$ is the sparsity level, then the solutions of the minimization problems in Theorems \ref{topologicalcondition} and \ref{P1-main} coincide (see Definition \ref{P0-P1-eq}) with overwhelming probability, as established in \cite{candes2006stable}.
\end{remark}

\begin{corollary}\label{gaussian}
Let $\mathcal{G}=(G,f,h)$ be a network dynamical system of size $N$ with maximum out-degree $\Delta(G)$, satisfying Assumptions \ref{isolated_ass} and \ref{assumption_about_coupling}. Let $\phi_q\in \mathbb{R}^{P\times N}$ be Gaussian random matrices whose entries are i.i.d. Gaussian with the mean zero and the variance $1/P$ where $P<N$, for each $q\in [N]$, such that $\Delta(G)+1 \leq  \frac{c_1P}{\log(N/P)}$ where $c_1>0$ is some positive constant. Assume that mean-fields
\begin{equation}
    y^q(1)=\phi_q x^q(1)
\end{equation}        
 are known for all $q\in[N]$. Then, the network $\mathcal{G}$ is uniquely topologically reconstructible with overwhelming probability via the minimization problem in equation \ref{P1}.
\end{corollary}

We establish bounds that allow one to determine the minimal number of mean-field measurements required, given the network's maximum degree.

\begin{corollary}
We summarize our findings as follows:
\begin{equation}
    \begin{split}
        &P > 2\Delta(G) + 1, \quad \text{(wTRC with full-spark matrices)};\\
        &\frac{c_1 P}{\log(N/P)} \geq \Delta(G) + 1, \quad \text{(sTRC w.h.p. with random Gaussian matrices)},
    \end{split}
\end{equation}
where $c_1$ is a positive constant, $P$ is the number of mean-field measurements, $N$ is the network size, and $\Delta(G)$ is the maximum out-degree of the underlying graph.
\end{corollary}

\begin{remark}
Analogous to Corollary~\ref{er_theorem}, we can obtain a sufficient condition for the Erd\H{o}s-Rényi graph \( G(N, p) \) when the mean-fields are generated with Gaussian random matrices as in Corollary \ref{gaussian}. In both the supercritical and connected regimes, the maximum degree \( \Delta(G) \) satisfies with high probability
\[
\Delta(G) < Np + \sqrt{2Np \log N}.
\]
Hence, the strong Topological Reconstruction Condition (sTRC) yields the bound:
\begin{equation}\label{P_theoretical}
Np + \sqrt{2Np \log N} + 1 < \frac{c_1P}{\log(N/P)}.
\end{equation}
This inequality provides a practical criterion for selecting the number of mean-field measurements \( P \) in terms of the network size \( N \) and edge probability \( p \). 
\end{remark}

\section{Full Network Reconstruction from Mean-Fields}\label{full-rec-sect}
We propose a two-stage method to reconstruct the network \( \mathcal{G} = (G, f, h) \) from mean-field observations. 
First, 
we uniquely recover the full state trajectory \( x^q(t) \).
%
Second, we infer the local dynamics at each vertex by fitting a linear combination of features derived from the recovered state trajectory. These features, referred to as the oracle dictionary, encode the exact functional form of each node’s evolution and will be introduced below. Although this stage resembles a model selection problem, it is constrained to operate on approximations of states recovered from mean-field observations.

\begin{definition}
Let \( x(t) \in \mathbb{R}^N \) denote the system state at time \( t \). The oracle dictionary for vertex \( i \in [N] \) is a set of spanning functions \( \{ \psi_i^l \}_{l=1}^{s_i} \), such that the local dynamics at node \( i \) are exactly described by
\begin{equation}
x_i(t+1) = \sum_{l=1}^{s_i} c_{i,l} \, \psi_i^l(x(t)),
\end{equation}
where \( c_{i,l} \in \mathbb{R} \) are nonzero coefficients. That is, the oracle dictionary provides an exact functional representation of the evolution of \( x_i \) in terms of the current system state. Let \( s_{\max} := \max_{i \in [N]} s_i \) denote the maximum number of terms across all vertices.
\end{definition}

By construction, the oracle dictionary guarantees that each $f$ and $h$ can be expressed exactly as a linear combination of the spanning functions in the dictionary. 

\begin{assumption}\label{ass:oracle_rank}
For each vertex \( i \in [N] \), define the dictionary matrix \( \Psi_i \in \mathbb{R}^{N s_{\max} \times s_i} \) as
\begin{equation}
\Psi_i := 
\begin{bmatrix}
\psi_i^1(x^1(0)) & \cdots & \psi_i^{s_i}(x^1(0)) \\
\psi_i^1(x^1(1)) & \cdots & \psi_i^{s_i}(x^1(1)) \\
\vdots & \ddots & \vdots \\
\psi_i^1(x^q(t)) & \cdots & \psi_i^{s_i}(x^q(t)) \\
\vdots & \ddots & \vdots \\
\psi_i^1(x^N(s_{\max}{-}1)) & \cdots & \psi_i^{s_i}(x^N(s_{\max}{-}1)) \\
\end{bmatrix},
\end{equation}
where each row corresponds to a specific pair \( (q, t) \), with \( q \in [N] \) indexing the perturbed initial conditions and \( t \in \{0, \dots, s_{\max} - 1\} \) indexing time.

We assume that this matrix has full column rank, i.e.,
\begin{equation} 
\mathrm{rank}(\Psi_i) = s_i,
\end{equation}
which ensures that the dictionary functions \( \{ \psi_i^l \}_{l=1}^{s_i} \) are linearly independent when evaluated across all trajectories and time steps.
\end{assumption}

\begin{theorem}\label{two-stage}
Let \( \mathcal{G} = (G, f, h) \) be a network dynamical system of size \( N \), with maximum out-degree \( \Delta(G) \), satisfying Assumptions~\ref{isolated_ass} and~\ref{assumption_about_coupling}.
Assume that for each node \( i \in [N] \), a set of basis functions \( \{ \psi_i^l \}_{l=1}^{s_i} \) is given which exactly represent the local dynamics, and that Assumption~\ref{ass:oracle_rank} holds. Let \( \phi_q \in \mathbb{R}^{P \times N} \) be measurement matrices for each \( q \in [N] \), satisfying
\begin{equation}
\delta_{2([\Delta(G)+1]^{s_{\max}})}(\phi_q) < \sqrt{2} - 1.
\end{equation}
Let the number of time steps be \( T = s_{\max} \), and assume access to mean-field measurements
\begin{equation}
\{ y^q(t) = \phi_q x^q(t) \}_{t = 1}^{s_{\max}}, \quad \text{for each } q \in [N].
\end{equation}
Then for each \( q \in [N] \), the state trajectory \( \{ x^q(t) \}_{t = 1}^{s_{\max}} \) is uniquely recoverable via the sequence of convex optimizations:
\begin{equation}\label{P11}
\min_{\tilde{x}^q(t)} \| \tilde{x}^q(t) \|_1 
\quad \text{subject to } \phi_q \tilde{x}^q(t) = y^q(t), 
\quad \text{for } t = 1, \dots, s_{\max}.
\end{equation}

Moreover:
\begin{enumerate}
    \item[(i)] The network topology is exactly reconstructible from \( \ \{supp(x^q(1))\}_{q=1}^N \), via the correspondence between \( L_1(q) \) and \( \supp(x^q(1)) \).
    \item[(ii)] The exact local dynamics at each vertex \( i \in [N] \) are recovered by solving 
    \begin{equation}
    x_i^{(q)}(1,\cdots,  s_{\max}) = \Psi_i \, c_i,
    \end{equation}
    where \( \Psi_i \in \mathbb{R}^{N s_{\max} \times s_i} \) is the oracle dictionary matrix, and the unique coefficients satisfy
    \begin{equation}
    \hat{c}_i = \Psi_i^\dagger x_i^{(q)}(1,\cdots,  s_{\max}), 
    \end{equation}
    where $ \Psi_i^\dagger$  denotes the Moore–Penrose pseudoinverse of \( \Psi_i \).
\end{enumerate}
\end{theorem}

We note that the assumption of access to an exact oracle dictionary is idealized. However, since we recover the level sets $L_1(q)$ uniquely, we can pinpoint the contribution of each term in the basis when the coupling $h$ and dynamics $f$ is provided.

\section{Numerical Experiments} \label{numerics-sec}

We present two topological reconstruction experiments to illustrate the Strong Topological Reconstruction Condition (sTRC). Afterwards, we present a full network reconstruction experiment.
To evaluate reconstruction performance, we use two metrics: the Matthews Correlation Coefficient (MCC) and the Mean Squared Error (MSE), defined below.

\begin{definition}
The \emph{Matthews Correlation Coefficient} (MCC) is a measure of the quality of binary classifications. It is defined as:

\begin{equation}
\mathrm{MCC} = \frac{ TP \cdot TN - FP \cdot FN }{ \sqrt{ (TP + FP)(TP + FN)(TN + FP)(TN + FN) } }
\end{equation}
where \( TP \), \( TN \), \( FP \), and \( FN \) denote the number of true positives, true negatives, false positives, and false negatives, respectively. The MCC ranges from \(-1\) to \(+1\), where \(+1\) indicates perfect prediction, \(0\) indicates performance no better than random guessing, and \(-1\) indicates total disagreement between prediction and ground truth.
\end{definition}

We also consider the following:

\begin{definition}
The \emph{Mean Squared Error} (MSE) quantifies the average squared difference between estimated values and ground truth values. It is defined as:

\begin{equation}
\mathrm{MSE} = \frac{1}{n} \sum_{i=1}^{n} (y_i - \hat{y}_i)^2
\end{equation}

where \( y_i \) denotes the true value, \( \hat{y}_i \) the predicted value, and \( n \) the number of samples. A lower MSE indicates more accurate predictions.
\end{definition}

\subsection{Topological Reconstruction From Mean-Field Measurements}

Our procedure consists of three stages: (1) generation of network states, (2) computation of mean-field measurements, and (3) solution of the inverse problem. 
The goal is to reconstruct the underlying network topology based on the mean-field measurements.

To evaluate performance, we apply this procedure across a range of graph sparsities and reconstruction thresholds, and compare the supports of the reconstructed vectors to the ground truth using the Matthews correlation coefficient (MCC). This provides a quantitative measure of reconstruction accuracy and robustness under varying conditions. The implementation details of the methodology are described below.

Using the setup of Section~\ref{setting-sec} our steps are:

\begin{enumerate}[label=(\roman*)]
    \item Each isolated dynamics $f_i : \mathbb{R} \rightarrow \mathbb{R}$ is chosen from the quadratic family:
    \begin{equation}
        f_i(x_i(t)) = r_i x_i(t)(1 - x_i(t)), \quad \forall i \in [N],
    \end{equation}
    where each $r_i$ is sampled uniformly at random from the set $\{1.2, 2.6, 3, 3.8\}$.
    
    \item The coupling function $h_{ij}: \mathbb{R} \times \mathbb{R} \rightarrow \mathbb{R}$ is given by
    \begin{equation}
        h_{ij}(x_i(t), x_j(t)) = \alpha_{ij} (x_i(t) - x_j(t)), \quad \forall i, j \in [N],
    \end{equation}
    where the matrix $\alpha = [\alpha_{ij}]$ is symmetric, and its upper triangular entries are sampled uniformly at random from $[0,1)$.
    
\item We generate undirected Erd\H{o}s-Rényi (ER) random graphs
$G_i(N, p_i)$ with $N = 1000$, where the edge probabilities $p_i$ are specified below.

    \item For each $G_i$, we generate  $\{x^q(1)\}_{q=1}^N$ by $q$-pinching the state $x^q(0)$, where perturbation size $\epsilon_q$ is sampled uniformly at random from $[0.5, 1)$.
    
    \item We generate $y^q(1) = \phi x^q(1)$, where $\phi$ is a $P \times N$ Gaussian random matrix with entries i.i.d. from $\mathcal{N}(0, \frac{1}{P})$, used uniformly across all $q \in [N]$.

    \item We solve Eq. (\ref{P1}) to reconstruct the topology.

    \item A threshold is applied to the reconstructed vectors obtained from Eq. (\ref{P1}), setting entries below the threshold to zero to isolate significant components.
    
    \item We extract the support set of each state vector and concatenate them. We concatenate the support sets obtained from Eq. (\ref{P1}).
    
\item After reconstruction, we use the Matthews Correlation Coefficient (MCC) between the true and reconstructed support sets.

\item Define \( P_c \) as the smallest value of \( P \) for which the MCC score exceeds 0.99.

\end{enumerate}
\subsubsection{Experiment 1: Supercritical Regime}

We investigate the transition behavior of the critical number of mean-field measurements $P_c$ required for successful topological reconstruction in the supercritical Erd\H{o}s-Rényi regime. We determine \( P_c \) by using a support recovery threshold of \(10^{-9}\). Our objective is to observe how the transition from reconstruction failure to success unfolds.

\begin{figure}[h!]
    \centering
    \includegraphics[width=1\textwidth]{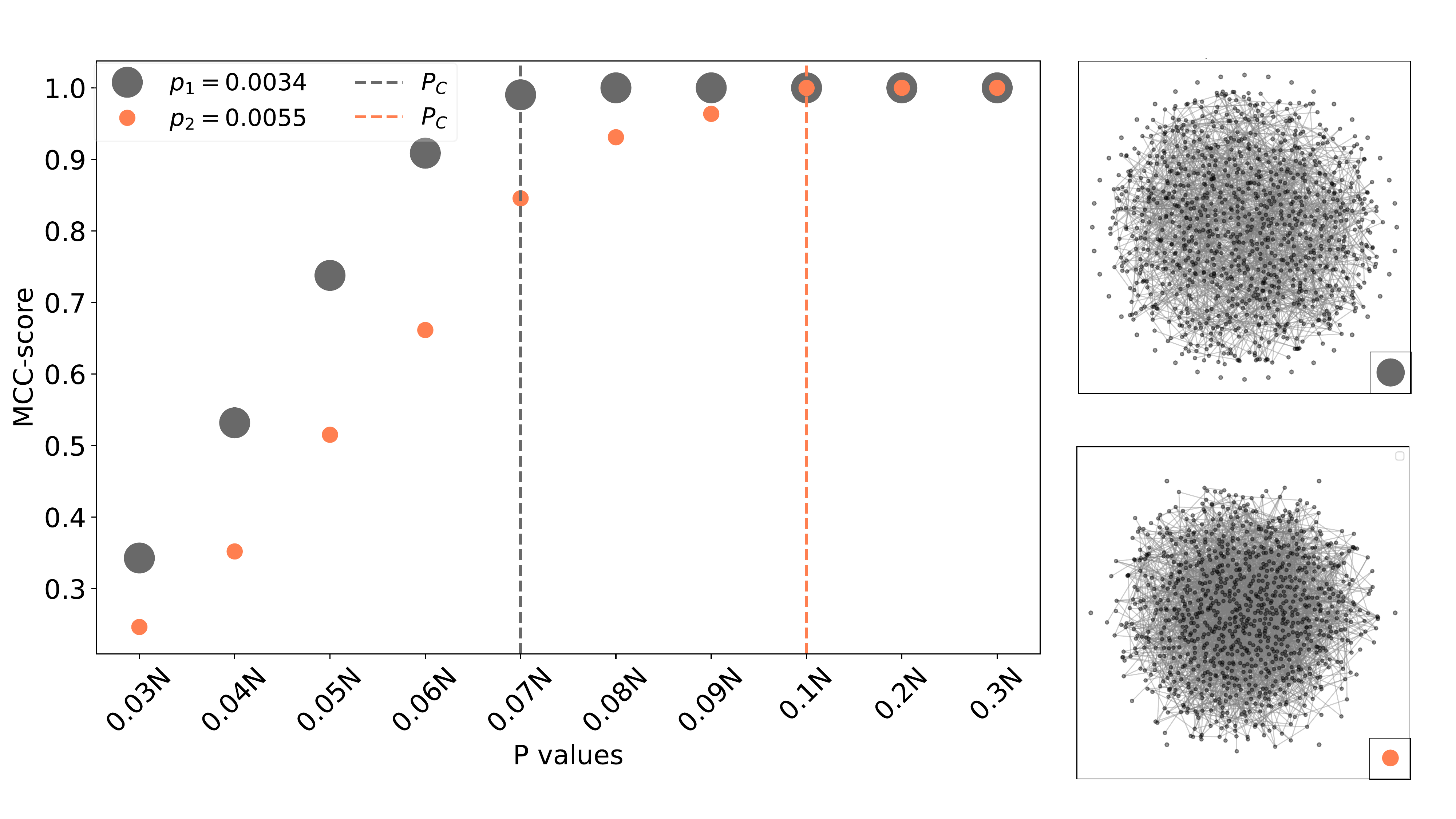}
\caption{
{\bf Critical number of mean-fields for the supercritical ER graphs}. ER graphs are generated with \( p_i = \frac{\log N}{N}(1 - \varepsilon_i) \), where \( \varepsilon_1 = 0.5 \) (in black) and \( \varepsilon_2 = 0.2 \) (in orange). 
The right insets depic the network topologies. In the left inset, we show the MCC-score and define $P_c$ as the critical number of mean fields that excedees $99\%$ of the score. For each case, the empirical critical number of mean-field measurements \( P_c \) is approximately 7\% and 10\% of the network size \( N \), respectively. 
}
\end{figure}

The fact that the required number of measurements \( P \) decreases as \( \varepsilon \) increases provides empirical support for the asymptotic behavior described in Corollary~\ref{er_theorem}. Furthermore, repeating the experiment with sine-function coupling and the same settings ($\varepsilon_2 = 0.2$, $N = 1000$, precision $10^{-10}$) shows that the results remain unchanged, with the transition occurring at the same point.
In the next experiment, we investigate this relationship, examining how the critical number of measurements \( P_c \) varies as a function of \( \varepsilon \) across a range of sparsity levels.

\subsubsection{Experiment 2: Supercritical to Connected Regime}
We investigate the monotonic behavior of the critical number of mean-field measurements required for successful topological reconstruction as we transition between the supercritical and connected Erd\H{o}s-Rényi regimes. We consider graphs where \( p = \frac{\log N}{N}(1 \pm \varepsilon) \) for \( \varepsilon \in (0,1) \), representing slightly sub- and super-connected regimes. We empirically determine the minimal number of mean-field measurements, denoted \( P_c \), at which reconstruction becomes exact, using two support recovery thresholds: \( 10^{-9} \) and \( 10^{-10} \). Our objectives are twofold: (i) to examine how \( P_c \) varies monotonically with the sparsity of the ER graphs, and (ii) to assess the robustness of the reconstruction method under different noise thresholds.

\begin{figure}[h!]
    \centering
    \includegraphics[width=0.8\textwidth]{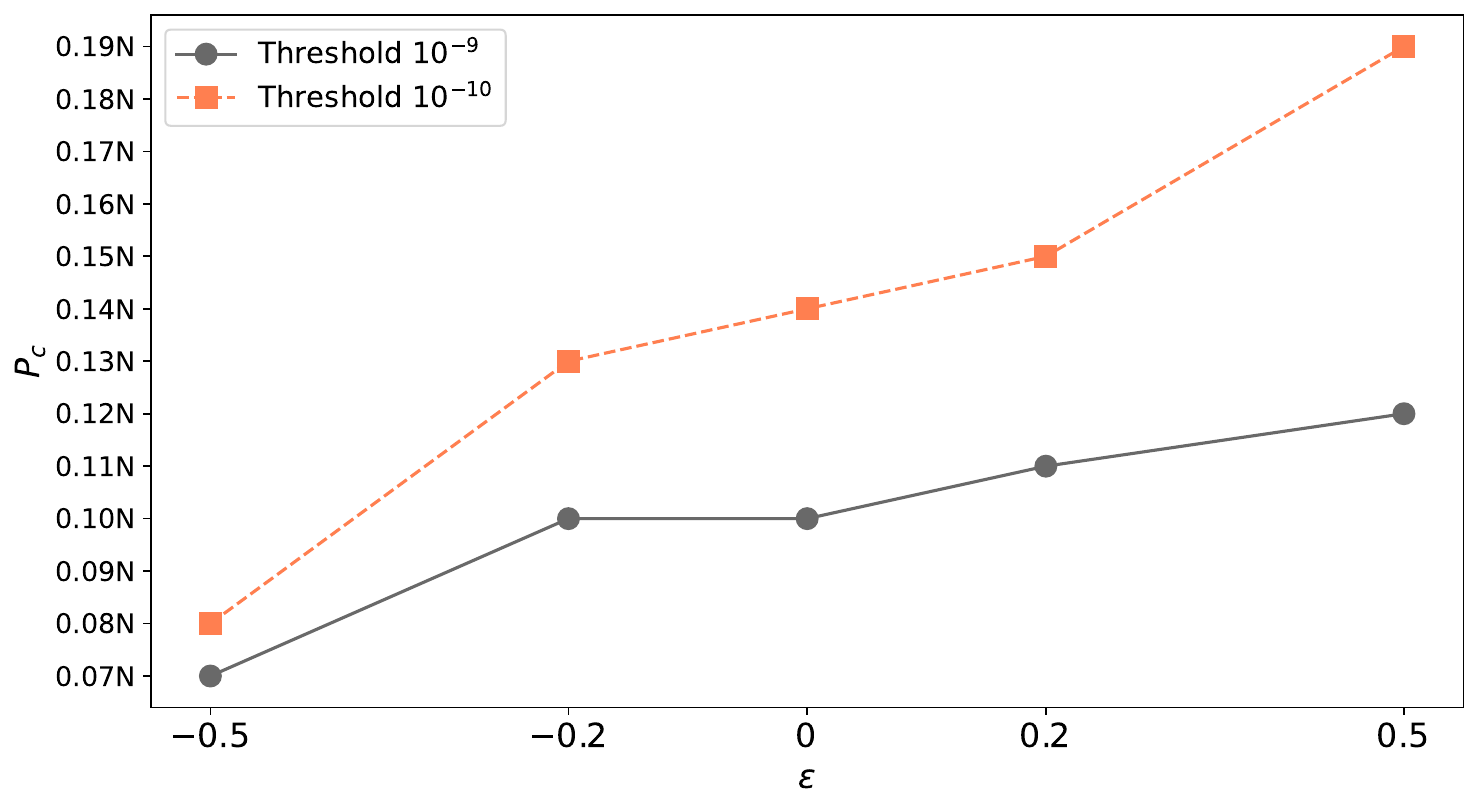}
\caption{{\bf Critical number of mean-fields across ER regimes}. We plot the empirical critical number of mean-field measurements \( P_c \) required for exact reconstruction, as a function of the sparsity parameter \( \varepsilon \), where \( p = \frac{\log N}{N}(1 + \varepsilon) \)  for a support recovery threshold of \( 10^{-9} \) and \( 10^{-10} \). 
This suggests that operating beyond the stricter \( P_c \) thresholds ensures robustness: if recovery is successful under a tighter tolerance (e.g., \(10^{-10}\)), it remains successful under looser ones (e.g., \(10^{-9}\)).
}
\end{figure}
The thresholds \( P_c \) reflect  the feasibility of exact support recovery and its robustness under relaxed accuracy requirements. 
A higher threshold defines a robust operating point where recovery remains accurate. Such robustness is valuable in practical scenarios where measurement noise, numerical error, or tolerance levels may vary.





\subsection{Full Network Reconstruction From Mean-Field Meaurements}

Our procedure consists of four main stages: (i) generation of network states, (ii) computation of mean-field measurements, (iii) solution of the inverse problem, and (iv) recovery of the underlying dynamics and topology from the reconstructed states. 
The goal is to reconstruct the underlying network \( \mathcal{G} = (G, f, h) \) from mean-field measurements. We perform the experiment with the linear map as isolated dynamics, random parameters, and a nondegenerate zero fixed point. This setting is sufficiently rich, since nonlinear perturbations that admit linearization can be locally approximated by the linear map. Using the setup of Section~\ref{setting-sec}, our steps are:

\begin{enumerate}[label=(\roman*)]
    \item Each isolated dynamic \( f_i : \mathbb{R} \rightarrow \mathbb{R} \) is defined as:
    \[
        f_i(x_i(t)) = r_i x_i(t), \quad \forall i \in [N],
    \]
    where each \( r_i \) is sampled uniformly at random from \( \{1.2, 2.6, 3, 3.8\} \).
    
    \item The coupling function \( h_{ij}: \mathbb{R} \times \mathbb{R} \rightarrow \mathbb{R} \) is given by:
    \[
        h_{ij}(x_i(t), x_j(t)) = \alpha_{ij} (x_j(t) - x_i(t)), \quad \forall i, j \in [N],
    \]
    where the matrix \( \alpha = [\alpha_{ij}] \) is symmetric, with its upper triangular entries sampled uniformly at random from \( [0,1) \).
    
    \item We generate an undirected Erd\H{o}s-Rényi (ER) random graph \( G(N, p) \) with \( N = 100 \), where the edge probability \( p \) is specified below.
    
    \item We generate the states \( \{\{x^q(t)\}_{q=1}^N\}_{t=1}^T \) by \( q \)-pinching initial state \( x^q(0) \), where the perturbation size \( \epsilon_q \) is sampled uniformly at random from \( [0.5, 1) \). 
    
    \item We fix the number of mean-field measurements at \( P = 0.6N \). For each \( t \in [T] \) and \( q \in [N] \), we generate:
    \[
        y^q(t) = \phi x^q(t),
    \]
    where \( \phi \) is a \( P \times N \) Gaussian random matrix with i.i.d. entries from \( \mathcal{N}(0, \frac{1}{P}) \), shared across all \( q \in [N] \).
    
    \item We solve Eq. (\ref{P1})  to reconstruct the states.

    \item A threshold is applied to the reconstructed vectors obtained from Eq. (\ref{P1}), setting entries below the threshold to zero to isolate significant components.
    
    \item We extract the support set of each state vector and concatenate them. We also concatenate the support sets obtained from Eq. (\ref{P1}).
    
    \item After reconstruction, we use the Matthews Correlation Coefficient (MCC) between the true and reconstructed support sets.
    
    \item The reconstructed trajectories \( \{\{\hat{x}^q(t)\}_{q=1}^N\}_{t=1}^T \)  capture the time evolution of each perturbed system under compressed sensing recovery.

\item For each \( t \in [T-1] \), we perform sparse regression over a linear polynomial feature library using the reconstructed subtrajectory \( \{\hat{x}^q(t')\}_{t'=0}^{t+1} \) for each \( q \in [N] \), thereby identifying the dynamics governing the temporal evolution from the trajectory data.

\end{enumerate}

\subsubsection{Experiment 3: Full Network Reconstruction}

We investigate the performance of the reconstruction using a threshold of \(10^{-9}\). We examine both the stability and success of the approach.
\begin{figure}[h!]
    \centering
    \includegraphics[width=1\textwidth]{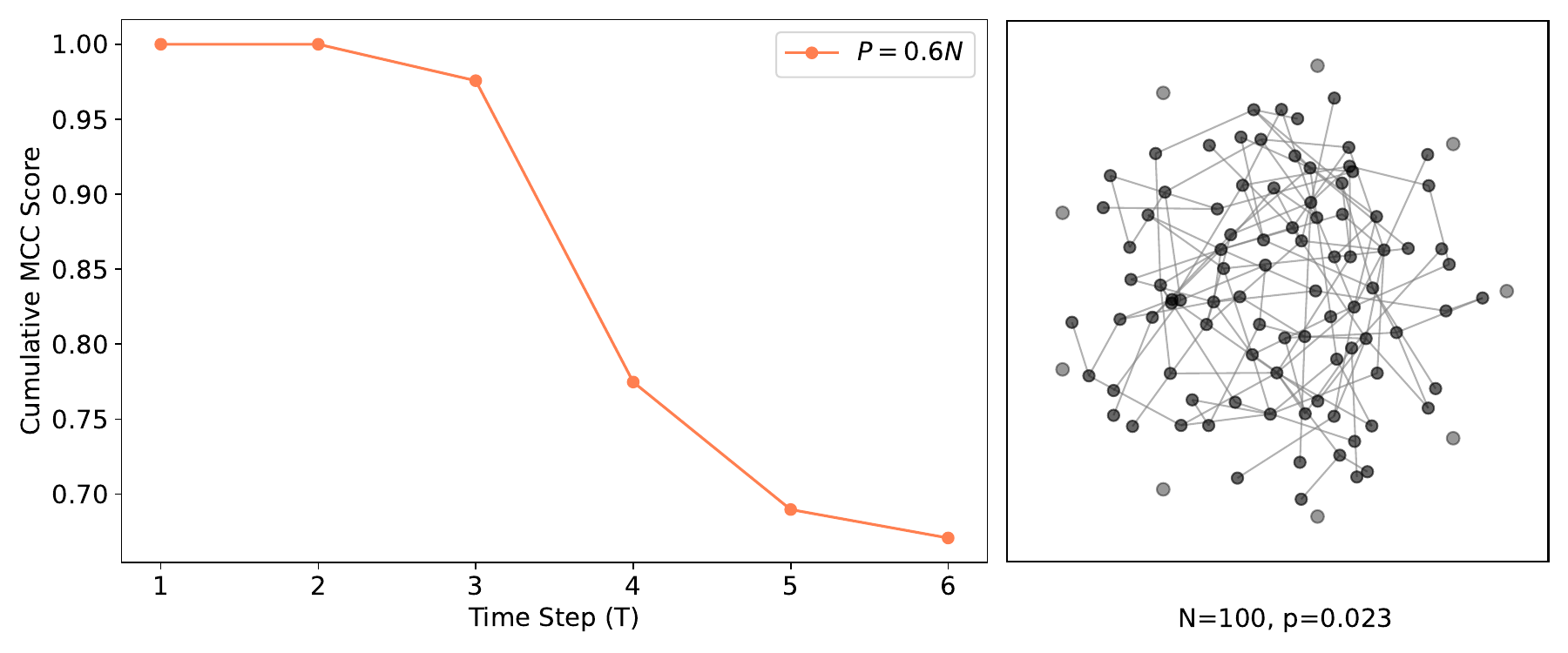}
    \caption{{\bf Cumulative MCC Score}. The ER graph on the right is generated with the given edge probability \( p \). The number of mean-field measurements is set to 60\% of the network size \( N \). The cumulative MCC score, computed from the reconstructed trajectories \( \{\{\hat{x}^q(t)\}_{q=1}^N\}_{t=1}^T \) up to each time step \( T \), is shown on the left.}
\end{figure}
\begin{figure}[h!]
    \centering
    \includegraphics[width=1\textwidth]{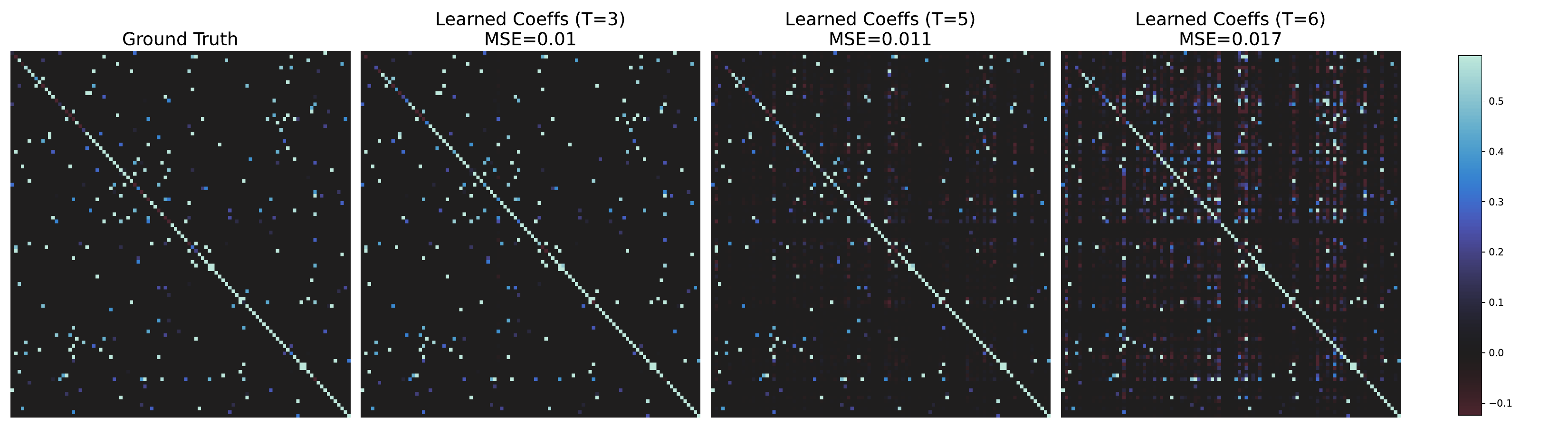}
    \caption{{\bf Heatmap for coefficient reconstruction}. The number of mean-field measurements is set to 60\% of the network size \( N \). The reconstructed trajectories \( \{\{\hat{x}^q(t)\}_{q=1}^N\}_{t=1}^T \), accumulated up to each time step \( T \), are used to generate the coefficient recovery. Each subplot corresponds to a different \( T \), with the respective MSE values shown above. The results demonstrate that recovery remains stable up to \( T = 5 \), consistent with Theorem~\ref{two-stage}.}\label{heatmap}
\end{figure}
When recovering coefficients from the reconstructed states, the twofold approach does not compromise stability, provided that the cumulative MCC score—which accounts for the reconstructed trajectories up to time \( T \)—does not drop significantly. The learned coefficients exhibit low MSE values up to \( T = 6 \). Reconstruction remains stable through \( T = 5 \), achieving an MSE of 0.011. At \( T = 3 \), the recovered coefficients yield an MSE of 0.01, as shown in Figure~\ref{heatmap}. These results align with Theorem~\ref{two-stage}. 
 We used only one-term isolated dynamics, and we would expect that as the number of active terms increases, accurate recovery would require either lower graph sparsity or a higher number of mean-field measurements.

\section{Conclusions}
We explored the reconstruction of hidden networks from mean-field measurements. A key limitation of our approach is the large number of required initializations, which reflects a fundamental trade-off in data acquisition: direct measurement of each individual vertex in large-scale systems is often infeasible. There are, however, potential strategies to reduce the number of pinchings. In the case of undirected graphs, the number of pinched nodes can be significantly reduced depending on the graph’s structure. An immediate reduction arises from the symmetry of the adjacency matrix. Furthermore, considering a second time iteration can lead to additional reductions. Notably, if probabilistic assumptions are permitted, the number of required pinchings can be substantially decreased. Nonetheless, determining the minimal number of initializations necessary for reliable reconstruction remains an open problem and a promising direction for future research.

Our framework  extends to cases where mean-field measurements are taken from an induced subgraph. This extension relies on the structural constraint that the measured subgraph has no incoming connections from the remainder of the network, only outgoing ones. As these outgoing interconnections do not alter the support of the measured states, they do not interfere with the reconstruction process.

\section*{Acknowledgments}
 We thank Edmilson Roque and Zheng Bian for enlightening discussions. This work was supported by the Serrapilheira Institute (Grant No.Serra-1709-16124), CNPq grant 312287/2021-6, and EPSRC-FAPESP Grant No. 2023/13706. Research carried out using the computational resources of the Center for Mathematical Sciences Applied to Industry (CeMEAI) funded by FAPESP (grant 2013/07375-0). The authors used ChatGPT (OpenAI, GPT-4, 2025) to assist with language editing. The code used to reproduce all numerical experiments in this paper is openly available at \cite{narci25trc}.

\section{Proofs} \label{proof-sec}
\subsection{Preliminaries}\label{prelim}
\begin{definition}
The support of the state vector $x \in \mathbb{R}^N $ is 
\begin{equation}
    \supp(x):=\{i:x_i\neq 0\}, 
\end{equation}
and the size of the support set is called $\ell_0$ norm (by abuse of notation), 
\begin{equation}
    \|x\|_0=\# \{i:x_i\neq 0\}.
\end{equation}
Moreover, a vector \(x\) is called \(s\)-sparse if $\|x\|_0\leq s$ for an integer \(s\).
\end{definition}

Given observations
\begin{equation}
    y = \phi x,
\end{equation}
where $y\in \mathbb{R}^P$, and $\phi\in \mathbb{R}^{P\times N}$ are known, our goal is to reconstruct $x\in \mathbb{R}^N$.

To define the unique reconstructability, we introduce the following optimization problem:

\begin{definition}\label{uniquely_reconstructable}
The following optimization problem
\begin{equation}
    x_* = \argmin_{\tilde{x} \in \mathbb{R}^N}  \|\tilde{x}\|_0 \quad \text{subject to}\quad \phi \tilde{x}=y,
\end{equation}
is called the $(P0)$ problem. 
The set of all solutions is denoted by $X_*$. A  vector $x$ is said to be $(P0)$ uniquely reconstructible if  $X_* = \{x\}$.
\end{definition} 
\begin{definition} 
We define $\ell_p$ norm of \(x\in \mathbb{R}^N\) as
\begin{equation}
    \|x\|_{p}= \big( \sum_{i=1}^N |x_i|^p\big)^{\frac{1}{p}}, \quad  1\leq p<\infty.
\end{equation}
\end{definition}

\begin{definition}
The following optimization problem
\begin{equation}
    \begin{split} &\min \|\tilde{x}\|_1 \\
&\text{subject to } \phi \tilde{x}=y, \end{split}
\end{equation}
is called the $(P1)$ problem and its solution is denoted by $x_*$.
\end{definition}

\subsection{Preparatory Results}

The Theorem \ref{topologicalcondition} and Theorem \ref{P1-main} will be obtained as a sequence of observations.
We use the notation $x^q(1)$ to represent the iterated state vector of the $q$-pinching initial condition $x^q(0)$. We denote the mean-fields as $y^q(1)$, where $y^q(1) = \phi_q x^q(1)$, and $\phi_q$ is the measurement matrix. 
   \begin{lemma}
    \label{pinch} {
$x^q_i(1)\neq 0$ when $i\neq q$ if and only if $i\in L_1(q)$.}
\end{lemma}
\begin{proof}
 First, we show that $i\in L_1(q)$ implies $x^q_i(1)\neq 0$.
Let $i \in L_1(q)$, then
\begin{equation}\begin{split}
    x^q_i(1)&= f_i(x^q_i(0)) + \sum_{j=1}^N A_{ij}h_{ij}(x^q_i(0), x^q_j(0))\\
    &=f_i(\delta_{iq}\epsilon_q)+ \sum_{j=1}^N A_{ij}h_{ij}(\delta_{iq}\epsilon_q, \delta_{jq}\epsilon_q), \; 
    \\
    &= 0+ A_{iq}h_{iq}(0,\epsilon_q),\\ 
    &= A_{iq}h_{iq}(0,\epsilon_q)\\
    & \neq 0. \nonumber
\end{split}    
\end{equation}
Observe that, the assumption $i \in L_1(q)$ means $A_{iq} \neq 0$ by Remark \ref{trc}, and $h_{iq}(0,\epsilon_q) \neq 0$  by Assumption \ref{assumption_about_coupling}. Thus, $x^q_i(1) \neq 0.$ For the other side of the implication, we want to show $x^q_i(1)\neq 0$ implies $i\in L_1(q)$. We use contraposition. Assume $i \notin  L_1(q).$ Then $A_{iq}=0$ and by the evolution law $x_i^q(1)=0.$ 
\end{proof}

\begin{lemma}\label{theta1}
    Let $D = \{ A \subseteq [N] \mid q \in A \}$ be the set of all subsets of $[N]$ that contain the element $q$, and let $C = \mathcal{P}([N] \setminus \{q\})$ be the power set of $[N] \setminus \{q\}$. Define  $\Theta_1^q : D \to C$ by
    \begin{equation}
        \Theta_1^q(A) = A \setminus \{q\}.
    \end{equation}
    Then, $\Theta_1^q$ is a one-to-one correspondence.
\end{lemma}

\begin{proof}
    We first show that $\Theta_1^q$ is injective. Suppose that for some $A, B \in D$, we have 
    \[\Theta_1^q(A) = \Theta_1^q(B).\]
    By the definition of $\Theta_1^q$, this means 
    \[ A \setminus \{q\} = B \setminus \{q\}. \]
    Since both $A$ and $B$ belong to $D$, they contain $q$, so we can conclude that $A = B$. Thus, $\Theta_1^q$ is injective.

    Next, we prove that $\Theta_1^q$ is surjective. Let $B \in C$, meaning $B$ is a subset of $[N] \setminus \{q\}$. Define $A = B \cup \{q\}$. Clearly, $A$ contains $q$, so $A \in D$, and we have
    \[ \Theta_1^q(A) = A \setminus \{q\} = B. \]
    Since $B$ was an arbitrary element of $C$, this shows that $\Theta_1^q$ is surjective.

    Since $\Theta_1^q$ is both injective and surjective, it is bijective (one-to-one correspondence).
\end{proof}

\begin{lemma}\label{havingL1} Let $x^q(1)$ be given for known $q \in [N]$. Then, there exists a one-to-one correspondence between $\supp(x^q(1))$ and $L_1(q)$:
    \begin{itemize}
    \item If \( x^q_q(1) \neq 0 \), the correspondence is given by the bijection $ \Theta_1^q: D\rightarrow C$ where $D=\{A\subseteq [N] \mid q\in A\}$ and  $C= \mathcal{P}([N]\setminus\{q\})$.
    \item If \( x^q_q(1) = 0 \), the correspondence is given by the identity map $ \mathrm{id}: \mathcal{P}([N]\setminus\{q\}) \to  \mathcal{P}([N]\setminus\{q\})$.
\end{itemize}
Furthermore, $|\supp(x^q(1))|\leq d_q+1$ where $d_q $ is the out-degree of the vertex $q$ of the underlying graph.
\end{lemma}

\begin{proof}
We have the state vector $x^q(1)$ for known $q\in[N]$.
We consider two  cases:

(a) The case \( x^q_q(1) \neq 0 \):

By Lemma \ref{pinch}, if $i \neq q$ and $x_i^q(1) \neq 0$, then $i \in L_1(q)$. By assumption, $x^q_q(1) \neq 0$, so we conclude that
\begin{equation}
    \supp(x^q(1)) = L_1(q) \cup \{q\}.
\end{equation}
Since $\supp(x^q(1)) \in D$ and $L_1(q) \in C$, where  $D=\{A\subseteq [N] \mid q\in A\}$ and  $C= \mathcal{P}([N]\setminus\{q\})$, the map $ \Theta_1^q: D\rightarrow C$ from Lemma (\ref{theta1}) is well-defined.

Applying \( \Theta_1^q \) to both sides, we compute:
\begin{equation}
\begin{split}
    \Theta_1^q (\supp(x^q(1)))
    &= (L_1(q)\cup \{q\})\setminus \{q\} \\
    &= L_1(q).
    \end{split}
\end{equation}
Since \( \Theta_1^q \) is a bijection, it establishes a one-to-one correspondence between $\supp(x^q(1))$ and $L_1(q)$ in this case.

(b) The case \( x^q_q(1) = 0 \):

In this case, $ q \notin \supp(x^q(1)) $, so we directly have
\begin{equation}
    \supp(x^q(1)) = L_1(q).
\end{equation}
Here, the identity map $ \mathrm{id}: \mathcal{P}([N]\setminus\{q\}) \to \mathcal{P}([N]\setminus\{q\})$ is a bijection.

Thus, for every $q$, we establish a one-to-one correspondence between $\supp(x^q(1))$ and $L_1(q)$, either via the bijection $\Theta_1^q$ in Case (a) or the identity map in Case (b). Furthermore, in both cases,  taking the cardinality of both sides gives
\begin{equation}
|\supp(x^q(1))|\leq |L_1(q)|+1.
\end{equation}
Since $|L_1(q)|=d_q$, by definition, we have
\begin{equation}
    |\supp(x^q(1))|\leq d_q+1.
\end{equation}
\end{proof}
Notably, there is no single bijection that applies to all $x^q(1)$ simultaneously. Instead, the reconstruction process is case-dependent.

\begin{corollary}    \label{corollary_reconstructable}
    If the state vector $x^q(1)$ is uniquely $(P0)$ reconstructible, then $L_1(q)$ can be uniquely determined via $(P0)$ problem. If $x^q(1)$ is uniquely $(P0)$ reconstructible for every $q\in [N]$,then the network $\mathcal{G}$ is topologically reconstructible. 
\end{corollary}
\begin{proof}
    This result follows from Lemma \ref{havingL1}. If the state vector \(x^q(1)\) is $(P0)$ uniquely reconstructible for a given $q\in [N]$, we can determine $L_1(q)$. Since \(\{L_1(q)\}_{q=1}^N\) uniquely determines the adjacency matrix \(A\) of the underlying graph of the network $\mathcal{G}$ (see Remark \ref{trc}), it follows that if the state vector \(x^q(1)\) is $(P0)$ uniquely reconstructible for every $q\in [N]$, then $\mathcal{G}$ is topologically reconstructible from mean-fields.
\end{proof}

\subsection{Proof of Theorem \ref{topologicalcondition}}
First, we need the following:
 \begin{lemma} \label{spark_sparsest}
    If the solution $x_*$ to $(P0)$ problem satisfies $\|x_*\|_0< \frac{\spark(\phi)}{2}$, then the solution is unique, and $x_*=x$.
    \end{lemma}
    \begin{proof}
    By contradiction, let  $x_1, x_2$ be distinct solutions. Then, $h=x_1-x_2$ is in the  null-space of $\phi$, meaning $\|h\|_0\geq \spark(\phi)$. Since $\|x_1\|_0, \|x_2\|_0< \frac{\spark(\phi)}{2}$, we get $\|h\|_0<\spark(\phi)$. 
\end{proof}

\noindent
{\bf Proof of Theorem \ref{topologicalcondition}}
The mean-field measurements are given by
\begin{equation}
    y^q(1) = \phi_q x^q(1),
\end{equation}
for each $q\in [N]$.

First, by Lemma \ref{spark_sparsest}, the optimization problem,
\begin{equation}
     \min\|\tilde{x}^q(1)\|_0 \text{  subject to } y^q(1) = \phi_q \tilde{x}^q(1),
\end{equation}
admits a unique solution $x^q_*(1)$ if
\begin{equation}
    \|x^q(1)\|_0<\frac{\spark(\phi_q)}{2}.
\end{equation}

Next, we apply the assumption that the measurement matrices $\{\phi_q(1)\}_{q=1}^N$ are full-spark $\spark(\phi_q)=P+1$. Also, by assumption, $P>2\Delta(G)+1$. Thus,
\begin{equation}
    \spark(\phi_q)> 2(\Delta(G)+1), \text{ for all }q\in [N].
\end{equation}
Then, the sufficient condition for $(P0)$ unique reconstructibility simplifies to 
\begin{equation}\label{reco_simple}
    \|x^q(1)\|_0<\Delta(G)+1, \text{ for all }q\in [N].
\end{equation}

Since Lemma \ref{havingL1} establishes that $\|x^q(1)\|_0\leq d_q+1$ for any $q\in [N]$, and by definition, the maximum out-degree of the network is
\begin{equation}
    \Delta(G) = \max_{q\in[N]}d_q,
\end{equation}
it follows that
\begin{equation}
    |\supp(x^q(1))|\leq \Delta(G)+1, \text{ for all }q\in [N].
\end{equation}
Consequently, by Corollary \ref{corollary_reconstructable}, the network $\mathcal{G}$ is topologically reconstructible.  $\square$

   \subsection{Proof of Corollary \ref{first_corol}}

A matrix \(\phi\) is full-spark if and only if every set of $P$ columns of $\phi$ is linearly independent, which is equivalent to every $P\times P $ submatrix of $\phi$ being invertible.
     
The set of matrices with $P$ linearly dependent columns (i.e., matrices where the determinant is zero) forms a lower-dimensional subset in the space of all matrices. Since the determinant is a continuous function of the matrix entries, the set of matrices where the determinant is zero has Lebesgue measure zero in the space of all possible matrices. Furthermore, a finite union of these determinant-zero sets remains a measure zero set.
     
Gaussian measure is absolutely continuous with respect to Lebesgue measure. This implies that the probability of selecting a matrix from any set of measure zero (such as the set where the determinant is zero) is zero. Therefore, the union of the  $\det^{-1}(0)$ sets, which corresponds to matrices with linearly dependent columns, has Gaussian measure zero.

Thus, with probability one, each $\phi_q$ is full-spark, and the corollary follows.  $\square$

\subsection{Proof of Corollary \ref{er_theorem}}

To formalize the scaling relationships used in the proof of Corollary \ref{er_theorem}, we introduce the following asymptotic notation:

\begin{definition}
Let \( f(P) \) and \( g(P) \) be functions from the integers to the positive real numbers. We define their asymptotic relations as follows:

\begin{itemize}
    \item \( f(P) = \mathcal{O}(g(P)) \) if there exist constants \( C > 0 \) and \( P^0 \in \mathbb{N} \) such that \( |f(P)| \leq C g(P) \) for all \( P > P^0 \).
    
    \item \( f(P) = \Theta(g(P)) \) if there exist constants \( C_1, C_2 > 0 \) and \( P^0 \in \mathbb{N} \) such that \( C_1 g(P) \leq |f(P)| \leq C_2 g(P) \) for all \( P > P^0 \).
    
\item \( f(P) = \Omega(g(P)) \) if there exist constants \( C_1 > 0 \) and \( P^0 \in \mathbb{N} \) such that \( |f(P)| \geq C_1 g(P) \) for all \( P > P^0 \).

    \item \( f(P) = o(g(P)) \) if for every \( \varepsilon > 0 \), there exists \( P^0 \in \mathbb{N} \) such that \( |f(P)| < \varepsilon g(P) \) for all \( P > P^0 \).
\end{itemize}
\end{definition}

\subsubsection{Proof of Corollary \ref{er_theorem}}
Let the maximum degree \( \Delta(G) \) deviate from the expected degree by a perturbation \( \xi(N) \), i.e.,
\begin{equation}
        \Delta(G) = Np + \xi(N).
\end{equation}
In the large \( N \) limit, the weak Topological Reconstruction Condition (wTRC) from Theorem \ref{topologicalcondition} becomes
\begin{equation}
    \begin{split}
    2Np + 2\xi(N) + 1 &< P(N), \\
    p &< \frac{P(N)}{2N} \left( 1 - \frac{1}{P(N)} - \frac{2\xi(N)}{P(N)} \right).
\end{split}
\end{equation}
Thus, to satisfy the bound asymptotically, the deviation must satisfy
\begin{equation}
    \xi(N) = o(P(N)).
\end{equation}

We now analyze this requirement in two classical regimes of the Erd\H{o}s-Rényi graph:

Case 1: (Supercritical regime) Suppose \( p = \frac{\log N}{N}(1 - \epsilon) \) for some fixed \( \epsilon \in (0,1) \). It is known that in this regime, the maximum degree satisfies
\begin{equation}
     \Delta(G) = \mathcal{O}\left( \frac{\log N}{\log\log N} \right)
\end{equation}
with high probability \cite{frieze2015introduction}. Hence, taking \( \xi(N) = \mathcal{O}\left( \frac{\log N}{\log\log N} \right) \), we find that if \( P(N) = \Omega(\log N) \), then
\begin{equation}
     \xi(N) = o(P(N)).
\end{equation}
Therefore, wTRC is satisfied with high probability, ensuring unique topological reconstruction.

Case 2: (Connected regime) Suppose \( p = \frac{\log N}{N}(1 + \epsilon) \) for some fixed \( \epsilon \in (0,1) \). Then the graph is connected with high probability, and it is well known (e.g., \cite{frieze2015introduction}) that:
\begin{equation}
        \mathbb{P}\left(\Delta(G) \leq \mathbb{E}[d] + \mathcal{O}(\sqrt{Np\log N}) \right) \to 1.
\end{equation}
In this case, since \( p = \Theta\left( \frac{\log N}{N} \right) \), we find
\begin{equation}
      \xi(N) = \mathcal{O}(\sqrt{\log^2 N}) = \mathcal{O}(\log N).
\end{equation}
Thus, if \( P(N) = \Omega(\log^2 N) \), we again have
\begin{equation}
     \xi(N) = o(P(N)).
\end{equation}
and wTRC holds with high probability.  $\square$

\subsection{Proof of Theorem \ref{P1-main}} 

We define  $(P0) \sim (P1)$ regime as follows: 
\begin{definition}\label{P0-P1-eq}
    The region in which the solution of the $(P1)$ problem is unique and it is identical to the solution to $(P0)$ problem is called the
$(P0)\sim (P1)$ equivalence regime.
\end{definition}

In the $(P0)\sim (P1)$ equivalence regime, the unique solution to $(P1)$ problem is determined, and it is shown to coincide with the unique solution of the $(P0)$ problem \cite{candes2005decoding}. This equivalence regime not only expands the region where the $(P0)$ problem admits a unique solution, but also makes this region computationally tractable.

To establish the Theorem \ref{P1-main}, we invoke the following lemma: 
\begin{lemma}\label{noiseless}
    Assume \(\delta_{2s}(\phi)<\sqrt{2}-1.\) Then the solution \(x_*\) to \( P_1\) problem obeys
    \begin{equation}
        \|x-x_*\|_2 \leq c_0 s^{-1/2}  \|x-x_s\|_1
    \end{equation}
for some constant \(c_0\) where \(x_s\) is the vector x with all but the \(s\)-largest entries set to zero. The recovery is exact if \(x\) is \(s\)-sparse \cite{candes2008restricted}.
\end{lemma}
\subsubsection{Proof of Theorem \ref{P1-main}}
The mean-field measurements are given by
\begin{equation}
    y^q(1) = \phi_q x^q(1),
\end{equation}
for each $q\in [N]$.
First, by Lemma \ref{noiseless}, the optimization problem,
\begin{equation}
     \min\|\tilde{x}^q(1)\|_1 \text{  subject to } y^q(1) = \phi_q \tilde{x}^q(1),
\end{equation}
admits a unique solution $x^q_*(1)$ if
\begin{equation}
    \|x^q(1)\|_0\leq s,
\end{equation}
where $s =\Delta(G)+1$.

Then, the sufficient condition for $(P1)$ unique reconstructibility simplifies to 
\begin{equation}\label{reco_simple}
    \|x^q(1)\|_0\leq \Delta(G)+1, \text{ for all }q\in [N].
\end{equation}

Since Lemma \ref{havingL1} establishes that $\|x^q(1)\|_0\leq d_q+1$ for any $q\in [N]$, and by definition, the maximum out-degree of the network is
\begin{equation}
    \Delta(G) = \max_{q\in[N]}d_q,
\end{equation}
it follows that
\begin{equation}
    |\supp(x^q(1))|\leq \Delta(G)+1, \text{ for all }q\in [N].
\end{equation}
Consequently, by Corollary \ref{corollary_reconstructable}, the network $\mathcal{G}$ is topologically reconstructible. 
 $\square$
\subsection{Proof of Theorem \ref{two-stage}}

The mean-field measurements are given by
\begin{equation}
    y^q(t) = \phi_q x^q(t),
\end{equation}
for each $q\in [N]$ and each $t\in [T]$.

First, by Lemma \ref{noiseless}, the optimization problem,
\begin{equation}
     \min\|\tilde{x}^q(t)\|_1 \text{  subject to } y^q(t) = \phi_q \tilde{x}^q(t),
\end{equation}
admits a unique solution $x^q_*(t)$ if
\begin{equation}
    \|x^q(t)\|_0\leq s,
\end{equation}
where $s =[\Delta(G)+1]^{s_{\max}}$.

Lemma \ref{havingL1} establishes that $\|x^q(1)\|_0 \leq d_q+1$ for any $q \in [N]$. Since the maximum out-degree is defined as
\begin{equation}
    \Delta(G) = \max_{q \in [N]} d_q,
\end{equation}
it follows that
\begin{equation}
    |\supp(x^q(1))| \leq \Delta(G)+1, \quad \text{for all } q \in [N].
\end{equation}
Since at each time step the support of $x^q(t)$ may grow by at most $\Delta(G)+1$, we conclude by induction that
\begin{equation}
    |\supp(x^q(t))| \leq [\Delta(G)+1]^t, \quad \text{for all } q \in [N], \, t \in [T].
\end{equation}
By assumption, $T = s_{\max}$, and so we obtain
\begin{equation}
    |\supp(x^q(t))| \leq [\Delta(G)+1]^{s_{\max}}, \quad \text{for all } q \in [N], \, t \in [T].
\end{equation}

 Thus, the recovered state vectors $x^q(t)$ are unique for every $q \in [N]$, and $t\in [T]$. 
 
The system evolves according to the oracle dictionary:
\begin{equation}
    x_i(t+1) = \sum_{l=1}^{s_i} c_{i,l} \, \psi_i^l(x(t)).
\end{equation}

Stacking the system over time, we obtain:
\begin{equation}
    \begin{pmatrix}
        x_i(1)\\
        x_i(2)\\
        \vdots\\
        x_i(T)
    \end{pmatrix} =
    \underbrace{
    \begin{pmatrix}
        \psi_i^1(x(0)) & \cdots & \psi_i^{s_i}(x(0)) \\
        \psi_i^1(x(1)) & \cdots & \psi_i^{s_i}(x(1)) \\
        \vdots         & \ddots & \vdots \\
        \psi_i^1(x(T-1)) & \cdots & \psi_i^{s_i}(x(T-1))
    \end{pmatrix}}_{\psi_i \in \mathbb{R}^{T \times s_i}}
    \begin{pmatrix}
        c_{i,1}\\
        \vdots\\
        c_{i,s_i}
    \end{pmatrix}.
\end{equation}

The coefficients $c_{i,l}$ are assumed to be non-zero by the definition of the oracle dictionary. For the coefficient vector to be uniquely identifiable, it is necessary that the matrix $\psi_i$ has full column rank. The necessary condition is that
\begin{equation}
    T \geq s_i, \quad \text{for each } i \in [N].
\end{equation}
Since $T = s_{\max}$ and $s_i \leq s_{\max}$, this condition is satisfied. Therefore, the coefficient vector $c_i = (c_{i,1}, \ldots, c_{i,s_i})^T$ is uniquely determined.

Thus, the maximum time step is the minimum value assuring the existence of the unique coefficients.

Observe that this is for any state vector, but we only have access to pinching states. So, we will use an enlarged dictionary matrix to suffice the unique coefficients.
 
 Item (i) follows from the correspondence in Theorem \ref{P1-main}.

Item (ii):
Given that each $\psi_i$ has full column rank and $Ns_{\max}>s_i$ for each $i\in[N]$, the pseudoinverse $\Psi_i^\dagger$ simplifies to the left inverse: 
\begin{equation}
    \Psi_i^\dagger = (\Psi_i^T\Psi_i)^{-1}\Psi_i^T.
\end{equation}
Since the quadratic loss in the associated least-squares problem is strictly convex (due to full rank), the solution
\begin{equation}
    \hat{c}_i = (\Psi_i^T\Psi_i)^{-1}\Psi_i^Tx_i^{(q)}(1,\cdots,s_{\max}),
\end{equation}
is unique.
 $\square$

 \bibliographystyle{plain}
\bibliography{references}
\end{document}